\documentclass[a4paper,12pt]{article}
\usepackage[most]{tcolorbox}
\newtcbox{\mybox}[1][]{enhanced, colframe=black, colback=orange!45, 
       nobeforeafter, tcbox raise base, shrink tight, extrude by=1mm, #1}
\usepackage{tikz} 
\usetikzlibrary{arrows.meta, positioning, calc} %
\usepackage{inputenc}
\usepackage{graphics}
\usepackage{epsfig}
\usepackage{graphicx}
\usepackage{multirow}
\usepackage{latexsym}
\usepackage{graphics}
\usepackage{multirow}
\usepackage{tikz}
\usepackage{epsfig}
\usepackage{amsmath,amssymb,amsfonts,theorem,color,bm}
\usepackage{multirow}
\usepackage{verbdef}
\usepackage{mathrsfs}
\usepackage{hyperref}
\usepackage{subcaption}
\usepackage{epstopdf}
\usepackage{float}
\newcommand{\beqn}{\begin{equation}}
\newcommand{\eeqn}{\end{equation}}

\newtheorem{remark}{Remark}[section]

\newtheorem{corollary}{Corollary}[section]
\newtheorem{theorem}{Theorem}[section]

\newtheorem{proposition}[theorem]{Proposition}

\def\blackbox{\leavevmode\vrule height 5pt width 4pt depth 0pt\relax}
\newenvironment{proof}{\begin{trivlist}
\item[]\hspace{0cm}{\bf Proof:}
\hspace{0cm} }{\hfill $\blackbox$
\end{trivlist}}

\definecolor{gray1}{gray}{0.25}

\newcommand\restr[2]{{
\left.\kern-\nulldelimiterspace 
#1 
\right|_{#2} 
}}

\newcommand{\mK}{\mathsf{K}}

\begin{document}
\title{Variably Scaled Kernels for the regularized solution of the parametric Fourier imaging problem}

\author{
	A. Volpara\thanks{MIDA, Dipartimento di Matematica, Università di Genova, Italy} 
	\quad A. Lupoli\thanks{Technical University of Munich and Munich Center
for Machine Learning, Germany}
    \quad E. Perracchione\thanks{Dipartimento di Scienze Matematiche \lq\lq Giuseppe Luigi Lagrange\rq\rq, Politecnico di Torino, Italy}\\
{\tt anna.volpara@edu.unige.it; lupa@ma.tum.de}\\ {\tt emma.perracchione@polito.it
}}

\maketitle

\section*{Abstract}
We address the problem of approximating parametric Fourier imaging problems via interpolation/ extrapolation algorithms that impose smoothing constraints across contiguous values of the parameter. Previous works already proved that interpolating via Variably Scaled Kernels (VSKs) the scattered observations in the Fourier domain and then defining the sought approximation via the projected Landweber iterative scheme, turns out to be effective. This study provides new theoretical insights, including error bounds in the image space and properties of the projected Landweber iterative scheme, both influenced by the choice of the scaling function, which characterizes the VSK basis. Such bounds then suggest a smarter solution for the definition of the scaling functions. Indeed, by means of VSKs, the information coded in an image reconstructed for a given parameter is transferred during the reconstruction
process to a contiguous parameter value. Benchmark test cases in the field of astronomical imaging, numerically show that the proposed scheme is able to regularize along the parameter direction, thus proving reliable and interpretable results. 
\section{Introduction}
Fourier imaging problems require the solution of the inverse Fourier transform problem when a limited number of noisy Fourier components of the imaging signal are at the disposal \cite{easton2010fourier,khare2015fourier,olson2017applied}. This image reconstruction problem is ill-posed in the sense of Hadamard, and therefore the use of regularization methods is recommended to address its solution \cite{engl1996regularization,bertero2006inverse}. For specific but significant applications, the image formation model is generalized to encode the fact that the imaging device may record sets of the Fourier samples for many different values of a specific parameter. Examples of this parametric Fourier imaging problem are dynamic Magnetic Resonance Imaging, where the parameter is time \cite{gamper2008compressed,sourbron2013classic,lingala2011accelerated,ravishankar2015efficient}, and imaging spectroscopy, where the parameter is energy \cite{prevost2022hyperspectral,zhang2022survey,gary2023new,piana2022hard}.

A possible approach to the Fourier imaging problem is based on interpolation-extrapolation schemes, in which an interpolation algorithm is applied in the spatial frequency domain and an out-of-band extrapolation approach is applied to the interpolated signal to obtain super-resolution effects \cite{larkman2007parallel,tsao2012mri,deng2012optimal,Massone1}. Recent studies \cite{perracchione2021feature,Perracchione2023} realized the interpolation step in the Fourier domain via Variably Scaled Kernels (VSKs) \cite{Bozzini1,vskmpi}, whose definition depends on a scale function. However, these studies did not prove (but claimed) that using the scaling function to mimic the target in the frequency space would lead to more accurate reconstructions in the image space. Further, as far as the parametric problem is concerned, an analysis based just on heuristic considerations and performed in the case of solar space imaging showed that the scale function can be leveraged to transfer knowledge from the image reconstructed at one value of the parameter to the contiguous one, thus realizing regularization along the parameter direction \cite{volpara2024ris}.

The present paper aims to fill the theoretical gaps concerned with these intuitions that, so far, has been validated just numerically. Specifically, the advances of the state of the art presented in our analysis are
\begin{enumerate}
    \item[(i)] \label{item1} $L_2$ error bounds (obtained thanks to the VSK Lebesgue functions) in the image space that depend on the selection of the scaling function in the Fourier domain;
    \item[(ii)] \label{item2} Bounds for the extrapolation iterative scheme that depend again on the selection of the scaling function; 
    \item[(iii)] \label{item3} New solutions for the definition of the VSK scaling function in the context of inverse problems that continuously depend on a parameter.
\end{enumerate}
In particular, thanks to item (i) we have been able to theoretically prove that the closer is the scaling function to the target in the Fourier domain, the smaller is the error in the physical space, which in turn provides useful insights to define a VSK basis tailored for the Fourier imaging problems. Further, the theoretical findings have been validated by considering samples from the {{Reuven Ramaty High Energy Solar Spectroscopic Imager (RHESSI)}} \cite{enlighten1658}, which was a telescope (now dismissed) recording X-rays from the Sun with the main purpose of observing solar flares \cite{piana2022hard}. Within this framework, imaging spectroscopy consists in providing reconstructions of hard X-ray maps at different count energy channels. An application performed using RHESSI data allowed us to confirm a theoretical results that proves the regularization effects obtained by sequentially exploiting the scale function to encode prior knowledge on the image to reconstruct at different values of the count energies.

The paper is organized as follows. Section \ref{sec:intro} is the core of the work and shows the theoretical studies about both error bounds that highlight the dependence of the VSK scaling function and results of the Landweber scheme for specific definitions of such a scaling function. Section \ref{application} confirms the theoretical findings by considering the RHESSI imaging system. Our conclusions are offered in Section \ref{conclusions}.

\section{The VSK interpolation/extrapolation procedure for inverse Fourier transform by limited data}
\label{sec:intro}
Given $g\in L^2(\mathbb{R}^d)\cap\mathcal{C}(\mathbb{R}^d)$ and $f\in L^2(\mathbb{R}^d)$ such that 
\begin{equation}\label{eq:000}
    \mathcal{F}f = g,
\end{equation}
where $\mathcal{F}:  L^2(\mathbb{R}^d) \to  L^2(\mathbb{R}^d)$ is the Fourier operator, we aim to recover $f$ given a few samples of the function $g$. As the proposed scheme is based on interpolating in the Fourier domain and then in extrapolating such approximation, the continuity of the function $g$ is required.  Moreover, for simplicity, until Subsection \ref{param_prob} we limit our attention to non-parametric problems, and we suppose that both $f$ and $g$ are continuous real-valued functions. Any extension to complex function is straightforward by considering the real and imaginary parts separately. 

\subsection{Error analysis for the VSK interpolation/extrapolation procedure}
\label{sec:error}

In practical scenarios, usually only a few measurements $\{g_i\}_{i=1}^{n} \in D \subseteq \mathbb{R}^d$, where $D$ is a compact set, are available. This is the motivation for studying  interpolation/extrapolation approaches. More formally, this means that the problem in \eqref{eq:000} is indeed a discrete Fourier problem, where where $\mathcal{F}= \mathcal{F}  \cdot \chi_D$.

As only scattered measurements $\{g_i\}_{i=1}^{n} \in D \subseteq \mathbb{R}^d$ of the function  $g$ are available,  one may think of interpolating the given data in order to get a continuous function in the left hand side of \eqref{eq:000}. This approach may fall in the general framework of interpolation/extrapolation algorithms, see e.g. \cite{Massone1,sacchi}. We then have to find a continuous interpolating function $P_{g}:=P(g): D \longrightarrow \mathbb{R}$ so that the interpolation conditions
\begin{equation}\label{interpolationproblem}
P_{g}(\boldsymbol{u}_i) = g_i, \quad i=1,\ldots,n,
\end{equation}
hold true. 
By using kernel bases, the interpolation problem \eqref{interpolationproblem} has a unique solution if $P_{g} \in \textrm{span} \{ \kappa_{\gamma}(\cdot,\boldsymbol{u}_i), \boldsymbol{u}_i \in D\}$, where $\kappa_{\gamma} :  D \times  D \longrightarrow \mathbb{R}$ is a strictly  positive definite and radial kernel and $\gamma>0$ is the so-called \emph{shape parameter}, see e.g. \cite{fornberg2011stable, fornberg2004stable} for particular instances concerning its selection. We also remark that for certain choices of the radial kernel $\kappa$ we can associate a continuous function $\varphi_{\gamma}: [0,+\infty)\longrightarrow \mathbb{R}$,  such that
\begin{equation*}
\kappa_{\gamma}(\boldsymbol{w},\boldsymbol{z})=\varphi_{\gamma}(\|\boldsymbol{w}-\boldsymbol{z}\|_2),
\end{equation*}
for all $\boldsymbol{w}, \boldsymbol{z} \in D$. 

In the following we omit the dependence of the kernel on the shape parameter and we now suppose to use the well-known Variable Scaling Kernels (VSK) in their basic form \cite{Bozzini1,vskmpi}. 
The idea behind the VSK method is to augment the dimensionality of the kernel space by introducing a function $\tilde{g}$, which serves as a rough approximation of $g$.


The approach involves an interpolant of the form
$$
P_{g}^{\tilde{g}} (\boldsymbol{u}) := P_{g} (\boldsymbol{u},\tilde{g}(\boldsymbol{u})) = \sum_{i=1}^n c_i \kappa((\boldsymbol{u}_i,\tilde{g}(\boldsymbol{u}_i)),(\boldsymbol{u},\tilde{g}(\boldsymbol{u}_i))),
$$
where $\boldsymbol{c}=(c_1,\ldots,c_n)^{\intercal}$ are so that 
$$
\mK \boldsymbol{c} = \boldsymbol{g}, \quad \textrm{with} \quad \boldsymbol{g}=(g_1,\ldots,g_n)^{\intercal}.
$$
and $\mK_{ij}= \kappa((\boldsymbol{u}_i,\tilde{g}(\boldsymbol{u}_i)),(\boldsymbol{u}_j,\tilde{g}(\boldsymbol{u}_j)))$.
This technique has been proven effective, as it captures both the underlying structure and the shape of the function $\tilde{g}$, facilitating the interpolation process.

Since the samples of $g$ are contained in \(D \subset \mathbb{R}^d\), and we lack information outside of this set, we extend \(P_{g}^{\tilde{g}}\) to $0$ in $\mathbb{R}^d\setminus D$. Specifically, at first, we define:

\[
\hat{P}_{g}^{\tilde{g}}(\boldsymbol{u}) :=
\left\{
\begin{array}{ll}
P_{g}^{\tilde{g}}(\boldsymbol{u}),  & \boldsymbol{u} \in D, \\
c, & \text{otherwise},
\end{array}
\right.
\]
being $c>0$. Then we set
\(P_{g}^{\tilde{g}}(\textbf{u}) = \hat{P}_{g}^{\tilde{g}} (\textbf{u}) \cdot \phi(\textbf{u}) \), with
\[
\phi(\boldsymbol{u}) =
\begin{cases}
1, & \boldsymbol{u} \in D, \\
\phi_1(\boldsymbol{u}),  & \boldsymbol{u} \in \mathbb{R}^d \setminus D.
\end{cases}
\]
Here, $\phi_1(\boldsymbol{u})$ is a smooth function that rapidly decreases to $0$, and such that the function $\phi$ is continuous in $\mathbb{R}^d$.
Regarding $\tilde{g}$, since there are no issues with its continuity, we extend it by zero-padding to all of $\mathbb{R}^d$, i.e.,
\[
\tilde{g}(\boldsymbol{u}) :=
\left\{
\begin{array}{ll}
\tilde{g}(\boldsymbol{u}),  & {\boldsymbol{u}\in D}, \\
0, & \mbox{otherwise}. 
\end{array}
\right.
\]

Let $\bar{f}$ be the function reconstructed by taking the interpolating function on the left-hand side of \eqref{eq:000} i.e., $\bar{f}$ is so that
$$
({\cal F}\bar{f})(\boldsymbol{u}) = P_{g}^{\tilde{g}} (\boldsymbol{u}). 
$$
Then we are able to show the following generic error estimate. 

\begin{theorem} Assuming that the continuous functions $f,\bar{f},g$ and  ${P}_{g}^{\tilde{g}}$ 
belong to $ L_2 (\mathbb{R}^d)$, the following stability estimate holds true
    \begin{align*}
        \|f-\bar{f}\|_{L_2 (\mathbb{R}^d)} \leq  \nu(\boldsymbol{u}) +\|g-{P}_{g}^{\tilde{g}} \|_{L_2 (\mathbb{R}^d \setminus D)},
    \end{align*}
    where 
    \begin{equation*}
     \nu(\boldsymbol{u}) = \left(  \sup_{\boldsymbol{u} \in \mathbb{R}^d} \left((g(\boldsymbol{u})-P_{g}^{\tilde{g}} (\boldsymbol{u}))^2 \mu(D) \right)^{1/2} \right).
    \end{equation*}
\end{theorem}
\begin{proof}
    \begin{align*}
        \|f-\bar{f}\|_{L_2 (\mathbb{R}^d)}  & = \|({\cal F}^{-1}g)-({\cal F}^{-1}P_{g}^{\tilde{g}} ) \|_{L_2 (\mathbb{R}^d)}\\ 
        & = \|g-P_{g}^{\tilde{g}} \|_{L_2 (\mathbb{R}^d)} \leq \|g-P_{g}^{\tilde{g}} \|_{L_2 (D)} + \|g-P_{g}^{\tilde{g}} \|_{L_2 (\mathbb{R}^d \setminus D)} \\
        & \leq \left( \sup_{\boldsymbol{u} \in \mathbb{R}^d} \left((g(\boldsymbol{u})-P_{g}^{\tilde{g}} (\boldsymbol{u}))^2 \mu(D) \right)^{1/2} \right) + \|g-P_{g}^{\tilde{g}} \|_{L_2 (\mathbb{R}^d \setminus D)}.
    \end{align*}
\end{proof}

\begin{remark}\label{remark1}
    Considering the above theorem, as in real applications we usually have information on a compact set of the frequency space, we are not usually able to bound $\|g-P_{g}^{\tilde{g}} \|_{L_2 (\mathbb{R}^d \setminus D)}$. 
    Nevertheless, it should be realistically negligible, as the function $g$ should decrease to zero smoothly and rapidly.
    Moreover, we can control the first term and, as expected we would like that the interpolant tends to the function $g$. The interesting fact is that we will now bound such an error by considering its dependence on the scaling function $\tilde{g}$.
\end{remark}

In order to link the interpolation error to the scaling function of the VSKs \cite{Perracchionepolito}, we recall that for any distinct points $\{\boldsymbol{u}_i\}_{i=1}^n$, there exist functions (known as Lagrange or cardinal bases) $\varphi_j \in \textrm{span} \{\kappa(\cdot,\boldsymbol{u}_j), j=1,\ldots,n\}$ so that the interpolant can be written as
\begin{align*}
    P_{g}^{\tilde{g}}(\boldsymbol{u}) = \sum_{j=1}^n g_j \varphi_j(\boldsymbol{u}).
\end{align*}
Then, we get that
    \begin{align} \label{eq:interp}
        \left|\left(g - P_{g}^{\tilde{g}}\right)(\boldsymbol{u})\right| & \leq \|g - P_{\tilde{g}}^{\tilde{g}}\|_{\infty}+ \| \tilde{\textbf{{g}}}-\textbf{g} \|_{\infty} \lambda(\boldsymbol{u}), 
    \end{align}
    where 
    \[
\lambda(\boldsymbol{u}) = \sum_{j=1}^{n} | \varphi_j (\boldsymbol{u}) |,
    \]
are known as Lebesgue functions. 
Indeed, by adding and subtracting $P_{\tilde{g}}^{\tilde{g}}$ in \eqref{eq:interp}, we get
    \begin{align}\label{eq:error}
    \left|\left(g - P_{g}^{\tilde{g}}\right)(\boldsymbol{u})\right| & \leq \left|\left(g - P_{\tilde{g}}^{\tilde{g}}\right)(\boldsymbol{u})\right|+\left|\left(P_{{g}}^{\tilde{g}} - P_{\tilde{g}}^{\tilde{g}}\right)(\boldsymbol{u})\right|.
    \end{align}
    Then, the second term of \eqref{eq:error} can be bounded as
\begin{align*}
\left|\left(P_{{g}}^{\tilde{g}} - P_{\tilde{g}}^{\tilde{g}}\right)(\boldsymbol{u})\right| & = \left|\left(P^{\tilde{g}} ({\tilde{g}}-g)\right)(\boldsymbol{u})\right| \\
& = \sum_{j=1}^n | ({\tilde{g}_j}-g_j) (\boldsymbol{u}) \varphi_j(\boldsymbol{u}) | \leq  \| \tilde{\textbf{{g}}}-\textbf{g} \|_{\infty} \lambda(\boldsymbol{u}).
\end{align*}

Equation \eqref{eq:interp} states that to get rid of the second term on the right-hand side of \eqref{eq:error}, we only need to set the nodal values of $\tilde{g}$ equal to those of $g$. Nevertheless, the first term in \eqref{eq:error}, suggests that the scaling function $\tilde{g}$ should be close to $g$ on the whole domain, as in this way $P_{\tilde{g}}^{\tilde{g}}$ should approach $P_{g}^{\tilde{g}}$ and $g$ in turn. 

We can formally observe that, as already claimed in previous papers, having a scaling function that mimics the target one could be helpful. Indeed, we are able to put some prior information into the kernel basis itself. For instance, if we have to reconstruct many images depending on a given continuous parameter that varies in time or that depends on the hardware instrument, we could use such information to reconstruct the image at the next  parameter range. To see this, one has to investigate how the scaling function depends on the signal formation. 

\subsection{VSK interpolation/extrapolation procedure for parametric imaging problems}
\label{param_prob}

For many kinds of signals, in physics we encounter pairs of functions related by an integral of the form
\begin{align}\label{eq:schr}
    \tilde{g}(\boldsymbol{u},\beta) = \int_{\beta}^{+\infty} w(\boldsymbol{u},b) \eta(\beta, b) \, db,
\end{align}
for a fixed $\boldsymbol{u} \in \mathbb{R}^d$, and where $\beta \in \mathbb{R}$ is a given parameter that might depend on the hardware instrumentation. The reason why such integral transforms are common to many problems relies in the fact that \eqref{eq:schr} generalizes the quantum equivalence of the position and momentum representations
of Schrödinger’s equation (when $\boldsymbol{u}$ represents the space coordinate, and $\eta$ is the Fourier kernel). In other words, $w$ and $\tilde{g}$ are equivalent representations of the system
under study. Hence, the general form of \eqref{eq:schr} can be seen as the common representation of the signal formation in many fields, e.g.,  interferometry, optics, radioastronomy,  nuclear medicine scattering and inverse transport theory. For a fixed $\boldsymbol{u} \in \mathbb{R}^d$, we now suppose that the real-valued functions $g,w$ and $\eta$ are continuous functions and that the product between $w$ and $\eta$ is a Riemann-integrable function. To formally investigate how the scaling function may be selected for parametric imaging problems, we prove the following result. 

\begin{proposition}
    \label{th:g-tildeg}
    Let $\eta$ be  a bounded and Lipschitz continuous function in $\beta$. Let $\eta,w$ and the product of the functions $\eta$ and $w$ be in $L_2 \cap L_1((\beta, +\infty))$. Let $\tilde{g}$ be the scaling function of the VSK interpolant, computed as in \eqref{eq:schr}. 
Let $\beta_1$ and $\beta_2$ be larger then $\beta \in \mathbb{R}$. If $\beta_1 \longrightarrow \beta_2$ then $|\tilde{g}(\boldsymbol{u},\beta_1)-{\tilde{g}}(\boldsymbol{u},\beta_2) |\longrightarrow 0$. 
\end{proposition}
\begin{proof}
Without any restriction, we suppose  $\beta_2>\beta_1$. Then for a fixed $\boldsymbol{u} \in \mathbb{R}^d$, we have that:
\begin{align*}
    & |\tilde{g}(\boldsymbol{u},\beta_1)-{\tilde{g}}(\boldsymbol{u},\beta_2)| = \left |\int_{\beta_1}^{\infty} w(\boldsymbol{u},b) \eta(\beta_1, b) \, db-\int_{\beta_2}^{\infty} w(\boldsymbol{u},b) \eta(\beta_2, b) \, db \right |\\
    & =  \left |\int_{\beta_1}^{\beta_2} w(\boldsymbol{u},b) \eta(\beta_1, b) \, db + \int_{\beta_2}^{\infty} w(\boldsymbol{u},b) \eta(\beta_1, b) \, db -\int_{\beta_2}^{\infty} w(\boldsymbol{u},b) \eta(\beta_2, b) \, db \right| \\
    &\leq  \int_{\beta_1}^{\beta_2} \left |w(\boldsymbol{u},b) \eta(\beta_1, b)  \right|db  +\int_{\beta_2}^{\infty} \left |w(\boldsymbol{u},b)  (\eta(\beta_1,b) -\eta(\beta_2,b) ) \right |db \\
    & \leq |\beta_1-\beta_2| \sup_{b \in [\beta_1,\beta_2]} \left( | \eta(\beta_1, b)| \right) \int_{\beta_1}^{\beta_2} \left |w(\boldsymbol{u},b) \right|db\\
    & + \sup_{b \in [\beta_2,\infty]} \left( | \eta(\beta_1,b) -\eta(\beta_2,b)| \right) \int_{\beta_2}^{\infty} \left |w(\boldsymbol{u},b) \right| db \\
    & = |\beta_1-\beta_2| \sup_{b \in [\beta_1,\beta_2]} \left( | \eta(\beta_1, b)| \right) \int_{\beta_1}^{\beta_2} \left |w(\boldsymbol{u},b) \right|db \\
    & +  \left(|\eta(\beta_1,b^*) -\eta(\beta_2,b^*)|\right) \int_{\beta_2}^{\infty} \left |w(\boldsymbol{u},b) \right|db  \\
    & \leq |\beta_1-\beta_2|  \sup_{b \in [\beta_2,\infty]} \left( | \eta(\beta_1, b)| \right) \int_{\beta_1}^{\beta_2} \left |w(\boldsymbol{u},b) \right|db  +  L |\beta_1-\beta_2| \int_{\beta_2}^{\infty} \left |w(\boldsymbol{u},b) \right|db  \\
     & =|\beta_1-\beta_2| \left( \sup_{[\beta_1,\beta_2]} \left( | \eta(\beta_1, b)| \right) \int_{\beta_1}^{\beta_2} \left |w(\boldsymbol{u},b) \right|db + L  \int_{\beta_2}^{\infty} \left |w(\boldsymbol{u},b) \right|db \right).
\end{align*}
In the above inequality chain we have used the Hölder inequality and the fact that, denoting by $T$ either the interval $ [\beta_1,\beta_2] \subset [\beta,\infty]$ or $ [\beta_2,\infty] \subset [\beta,\infty]$, for continuous and bounded functions we have that
\begin{equation*}
    ||\eta(\beta_1, b)||_{L_{\infty}(T)} = {\rm ess} \sup_{b \in T} \left( | \eta(\beta_1, b)| \right) = \sup_{b \in T} \left( | \eta(\beta_1, b)| \right) < + \infty .
\end{equation*}

Then, as $w$ is in $L_1((\beta, +\infty))$, i.e., $||w||_{L_1((\beta, +\infty))}<+\infty$, also $ \int_{\beta_1}^{\beta_2}  |w(\boldsymbol{u},b)| db$ and $\int_{\beta_2}^{\infty} \left |w(\boldsymbol{u},b) \right|db$ are both bounded.  
Hence, all quantities in the last step of the inequality chain are bounded and we can affirm that if $|\beta_1-\beta_2| \longrightarrow 0$ then 
$|\tilde{g}(\boldsymbol{u},\beta_1)-{\tilde{g}}(\boldsymbol{u},\beta_2) |\longrightarrow 0$.
\end{proof}

The above result offers the opportunity to define a scaling function in the VSKs setting which depends on the parameter $\beta$. Hence, supposing to have some information on the sought images in the frequency plane for a given parameter $\beta_2$, we can use it to define the scaling function for reconstructing the target at a sufficiently close parameter $\beta_1$, thus obtaining smoothing constraints in successive reconstructions. In view of this, we now need to investigate how this $\beta$-dependent scaling function affects the inversion step which is carried out via the so-called Landweber iterative scheme. 

\subsection{Analysis of the Landweber scheme}\label{Landweber}

Due to the fact that in applications we would have scattered data only in a compact domain, once the interpolation step has been performed, the plain application of the inverse Fourier transform on the interpolated data would imply significant drawbacks, as ringing effects on the recovered image \cite{lagendijk1988regularized}.
 A possible way to overcome such limitations consists in using a constrained inversion techniques. More formally, if we denote by \(P_{g}^{\tilde{g}} \cdot \chi_D \) the interpolant on $D$ of the scattered observations, then 
\begin{equation*}
    \bar{f}_{\textrm{nb}} ({\boldsymbol{x}}) =  \left({\cal F}^{-1} (P_{g}^{\tilde{g}} \cdot \chi_D \right)) ({\boldsymbol{x}}),
\end{equation*}
is the noisy bandlimited approximation of the function ${f}$. 

For functions $f\in \mathcal{C}_+$, where $\mathcal{C}_+:= \{f\in L^2(\mathbb{R}^d) \ : \ f\ge0 \ \text{a.e}\}$ denotes the closed convex cone of non-negative functions in $L^2(\mathbb{R}^d)$, stable approximation can be computed via the projected Landweber iteration according to the following steps:
\begin{enumerate}
\item Initialize the iterative scheme as ${f}^{(0)} = 0$.
\item For a given threshold $\tau \in(0,2) $ and starting from $f^{(k)}$, compute
\begin{equation*}\label{projection-3}
    {\cal F}({f}^{(k+1)}) =  \tau {\cal F}(\Bar{f}_{\textrm{nb}} ) + (1- \tau  \cdot \chi_D) {\cal F}(\Tilde{f}^{(k)}), \quad k=1,2,\ldots.
\end{equation*}
and apply the non-negativity constraint by means of the projection
\begin{equation*}\label{projection-4}
\tilde{f}^{(k+1)} = {\cal{P}}_+f^{(k+1)},
\end{equation*}
where
\begin{equation*}
    \left({\cal{P}}_+ f^{(k+1)} \right) (\boldsymbol{x})  = \left \{ 
    \begin{array}{cc}
         0, &   \textrm{if} \hskip 0.2cm  \Re{(f^{(k+1)}(\boldsymbol{x}))}<0 ,\\
        f^{(k+1)}(\boldsymbol{x}), & \textrm{otherwise}.
    \end{array} \right.
\end{equation*}
\item  Continue with step 2. until $$\|\tilde{f}^{(k+1)}-\tilde{f}^{(k)}\|_{L^2(\mathbb{R}^d)}\le \delta,$$ where $\delta$ is a given tolerance. Therefore we consider $\tilde{f}:= \tilde{f}^{(k+1)}$ as the approximated solution of $f$.
\end{enumerate}

\begin{remark}
{We highlight that the procedure described above is a \emph{modified} Landweber method,  since it is not a fixed-point method, but depends on the initial value $\bar{f}_{nb}$ (from which the convergence depends). In addition, due to additional projections, iterations may lose the band-limited property even starting from an initially band-limited function. Finally, since operationally we do not have information outside the domain $D$, unlike the standard Landweber method which guarantees convergence throughout $\mathbb{R}^d$, this method can guarantee convergence towards the exact solution only in  the spectral domain $D$.}
\end{remark}

In order to investigate the influence of the $\beta$-dependent definition of the VSK scaling function, a crucial point consists in determining how the final approximated solutions are linked to the interpolants. 

\begin{theorem}\label{th:convergence}
Let $D$ a compact set in $\mathbb{R}^d$, and $\mathcal{P}_{-}$ be defined as:
\[
    \left({\cal{P}}_{-} f^{(k+1)} \right) (\boldsymbol{x})  = \left \{ 
    \begin{array}{cc}
          f^{(k+1)}(\boldsymbol{x}),  &   \textrm{if} \hskip 0.2cm \Re{(f^{(k+1)}(\boldsymbol{x}))}<0 ,\\
      0, & \textrm{otherwise}.
    \end{array} \right.
\]
    For a fixed $\beta$, and a given tolerance $\delta>0$, then  $$\left|\|\mathcal{F}\Bar{f}_{nb}-\mathcal{F}\tilde{f}^{(k)}\|_{L^2(D)}-\|\mathcal{F}\mathcal{P}_-f^{(k+1)}\|_{L^2(\mathbb{R}^d)}\right|\le \delta.$$
\end{theorem}
\begin{proof}
    Let us suppose  that $\|\tilde{f}^{(k+1)}-\tilde{f}^{(k)}\|_{L^2(\mathbb{R}^d)}\le \delta,$ and denoting by $\tilde{\tilde{f}}:= \mathcal{P}_-f,$ we get that
    \begin{align*}
        &\delta \ge \|\tilde{f}^{(k+1)}-\tilde{f}^{(k)}\|_{L^2(\mathbb{R}^d)}
        = \|\mathcal{F}f^{(k+1)}-\mathcal{F}\tilde{f}^{(k)}-\mathcal{F}\tilde{\tilde{f}}^{(k+1)}\|_{L^2(\mathbb{R}^d)}\\ 
        &= \|\tau\mathcal{F}\Bar{f}_{nb}  -\tau\chi_D\mathcal{F}\tilde{f}^{(k)}-\mathcal{F}\tilde{\tilde{f}}^{(k+1)}\|_{L^2(\mathbb{R}^d)} \\
        & \ge \left|\tau\|\mathcal{F}\Bar{f}_{nb}-\mathcal{F}\tilde{f}^{(k)}\|_{L^2(D)}-\|\mathcal{F}\tilde{\tilde{f}}^{(k+1)}\|_{L^2(\mathbb{R}^d)}\right|.
    \end{align*}
\end{proof}

We are now interested in understanding how the maps recovered considering contiguous parameters are linked to their reconstruction via Landweber.  

\begin{proposition}\label{prop land}
    Given two real parameters $\beta_1 < \beta_2$, we have that  
    \begin{align*}
    \|\Tilde{f}_{\beta_2}-\Tilde{f}_{\beta_1}\|_{L^2(\mathbb{R}^d)} \le &  \|\mathcal{F}f_{\beta_2}^{(k+1)}-\mathcal{F}f_{\beta_1}^{(k+1)}\|_{L^2(\mathbb{R}^d \setminus D)}  + \\
    & \left(1-(1-\tau)^{k+1}\right) \|\mathcal{F}\bar{f}_{nb,\beta_2}-\mathcal{F}\bar{f}_{nb,\beta_1}\|_{L^2(D)} 
    \\
    &+\sum_{l=0}^{k-1} (1-\tau)^{\ell +1}\|\mathcal{F}\tilde{\tilde{f}}_{\beta_2}^{(k-l)}-\mathcal{F}\tilde{\tilde{f}}_{\beta_1}^{(k-l)}\|_{L^2(D)}.
\end{align*}
\end{proposition}
\begin{proof}
By the iterative scheme, we have that
\begin{align*}
    \|\Tilde{f}_{\beta_2}-\Tilde{f}_{\beta_1}\|_{L^2(\mathbb{R}^d)} & = \|\Tilde{f}_{\beta_2}^{(k+1)}-\Tilde{f}_{\beta_1}^{(k+1)}\|_{L^2(\mathbb{R}^d)} \le \|f_{\beta_2}^{(k+1)}-f_{\beta_1}^{(k+1)}\|_{L^2(\mathbb{R}^d)}\\ 
    & = \|\mathcal{F}f_{\beta_2}^{(k+1)}-\mathcal{F}f_{\beta_1}^{(k+1)}\|_{L^2(\mathbb{R}^d)} \\
    & \le \|\mathcal{F}f_{\beta_2}^{(k+1)}-\mathcal{F}f_{\beta_1}^{(k+1)}\|_{L^2(\mathbb{R}^d \setminus D)} +\|\mathcal{F}f_{\beta_2}^{(k+1)}-\mathcal{F}f_{\beta_1}^{(k+1)}\|_{L^2(D)} \\
    & \le \|\mathcal{F}f_{\beta_2}^{(k+1)}-\mathcal{F}f_{\beta_1}^{(k+1)}\|_{L^2(\mathbb{R}^d \setminus D)} +  \tau \|\mathcal{F}\bar{f}_{nb,\beta_2}-\mathcal{F}\bar{f}_{nb,\beta_1}\|_{L^2(D)} \\
    & +(1-\tau)\|\mathcal{F}\tilde{f}_{\beta_2}^{(k)}-\mathcal{F}\tilde{f}_{\beta_1}^{(k)}\|_{L^2(D)}
    \\
    & \le \|\mathcal{F}f_{\beta_2}^{(k+1)}-\mathcal{F}f_{\beta_1}^{(k+1)}\|_{L^2(\mathbb{R}^d \setminus D)} +  \tau \|\mathcal{F}\bar{f}_{nb,\beta_2}-\mathcal{F}\bar{f}_{nb,\beta_1}\|_{L^2(D)} \\
    &  +(1-\tau)\|\mathcal{F}f_{\beta_2}^{(k)}-\mathcal{F}f_{\beta_1}^{(k)}\|_{L^2(D)} +(1-\tau)\|\mathcal{F}\tilde{\tilde{f}}_{\beta_2}^{(k)}-\mathcal{F}\tilde{\tilde{f}}_{\beta_1}^{(k)}\|_{L^2(D)}  
    \\  
 &      = \|\mathcal{F}f_{\beta_2}^{(k+1)}-\mathcal{F}f_{\beta_1}^{(k+1)}\|_{L^2(\mathbb{R}^d \setminus D)} +  \tau \|\mathcal{F}\bar{f}_{nb,\beta_2}-\mathcal{F}\bar{f}_{nb,\beta_1}\|_{L^2(D)} 
 \\ 
 &  +(1-\tau) \tau \| \mathcal{F}\bar{f}_{nb,\beta_2}-\mathcal{F}\bar{f}_{nb,\beta_1}\|_{L^2(D)} +(1-\tau)\|\mathcal{F}\tilde{\tilde{f}}_{\beta_2}^{(k)}-\mathcal{F}\tilde{\tilde{f}}_{\beta_1}^{(k)}\|_{L^2(D)}  \\
 &  + (1-\tau)^2\| \mathcal{F} \tilde{f}^{(k-1)}_{\beta_2}- \mathcal{F} \tilde{f}^{(k-1)}_{\beta_1}\|_{L^2(D)}
 \\  
    & \le \dots \le \|\mathcal{F}f_{\beta_2}^{(k+1)}-\mathcal{F}f_{\beta_1}^{(k+1)}\|_{L^2(\mathbb{R}^d \setminus D)}+ \\ & \left(\tau \sum_{l = 0}^{k} (1-\tau)^l\right) \|\mathcal{F}\bar{f}_{nb,\beta_2}-\mathcal{F}\bar{f}_{nb,\beta_1}\|_{L^2(D)} \\
    & 
    +\sum_{l=0}^{k-1} (1-\tau)^{\ell +1}\|\mathcal{F}\tilde{\tilde{f}}_{\beta_2}^{(k-l)}-\mathcal{F}\tilde{\tilde{f}}_{\beta_1}^{(k-l)}\|_{L^2(D)}\\
    & = \|\mathcal{F}f_{\beta_2}^{(k+1)}-\mathcal{F}f_{\beta_1}^{(k+1)}\|_{L^2(\mathbb{R}^d \setminus D)}  +\\ & \left(1-(1-\tau)^{k+1}\right) \|\mathcal{F}\bar{f}_{nb,\beta_2}-\mathcal{F}\bar{f}_{nb,\beta_1}\|_{L^2(D)} 
    \\
    & +\sum_{l=0}^{k-1} (1-\tau)^{\ell +1}\|\mathcal{F}\tilde{\tilde{f}}_{\beta_2}^{(k-l)}-\mathcal{F}\tilde{\tilde{f}}_{\beta_1}^{(k-l)}\|_{L^2(D)}.
\end{align*}
\end{proof}

{Proposition \ref{prop land}} states that we can bound the difference between solutions at different parameter values by the $L^2$ norm of the interpolants in  $D$. Specifically, if we use a scaling function defined thanks to contiguous values of $\beta$, we are essentially modeling our approximation by imposing constraints, i.e. ensuring a lower variation norm. 

Considering the above proposition, as stated in Remark \ref{remark1}, in real applications we usually have no information outside a compact Fourier domain, so we reasonably suppose that  $\|\mathcal{F}f_{\beta_2}^{(k+1)}-\mathcal{F}f_{\beta_1}^{(k+1)}\|_{L^2(\mathbb{R}^d \setminus D)}$ is \emph{small}. Nevertheless, we can bound the second and third term, inside the compact Fourier domain, and as far as the last term is concerned, we again expect that its contribution is negligible outside $D$.  

\section{Applications to astronomical imaging}
\label{application}

In this section, we focused on solar hard X-ray imaging and, specifically, we test
our interpolation/extrapolation procedure in the framework of the NASA RHESSI mission (\cite{enlighten1658}). 

From now on, we will denote the function $g$ with $V$ and $f$ with $I$, accordingly to the notations typically applied in this astronomical imaging scenario. The main purpose in the RHESSI system is to approximate the image of the unknown source flux emitted during a solar flare given its the Fourier transform, namely $V$, i.e. the imaging problem reads as: 
\begin{equation*}\label{map-visibilities}
 V(\boldsymbol{u}) =  ({\cal F} I )(\boldsymbol{u}) =\int_{\mathbb{R}^2}  I(\boldsymbol{x})  {\rm e}^{2\pi \imath(\boldsymbol{x}\cdot \boldsymbol{u})} \,d\boldsymbol{x} \, .
\end{equation*}
More precisely, the data provided by RHESSI consists of a set of scattered observations, called visibilities, associated with experimental measurements of the Fourier transform of the incoming photon flux, sampled at specific points (associated to the hardware detectors) in the spatial frequency plane lying on circles characterized by increasing radii (see Figure \ref{fig:rhessicoverage}). 
\begin{figure}[htbp]
    \centering
    \includegraphics[width=0.3\linewidth]{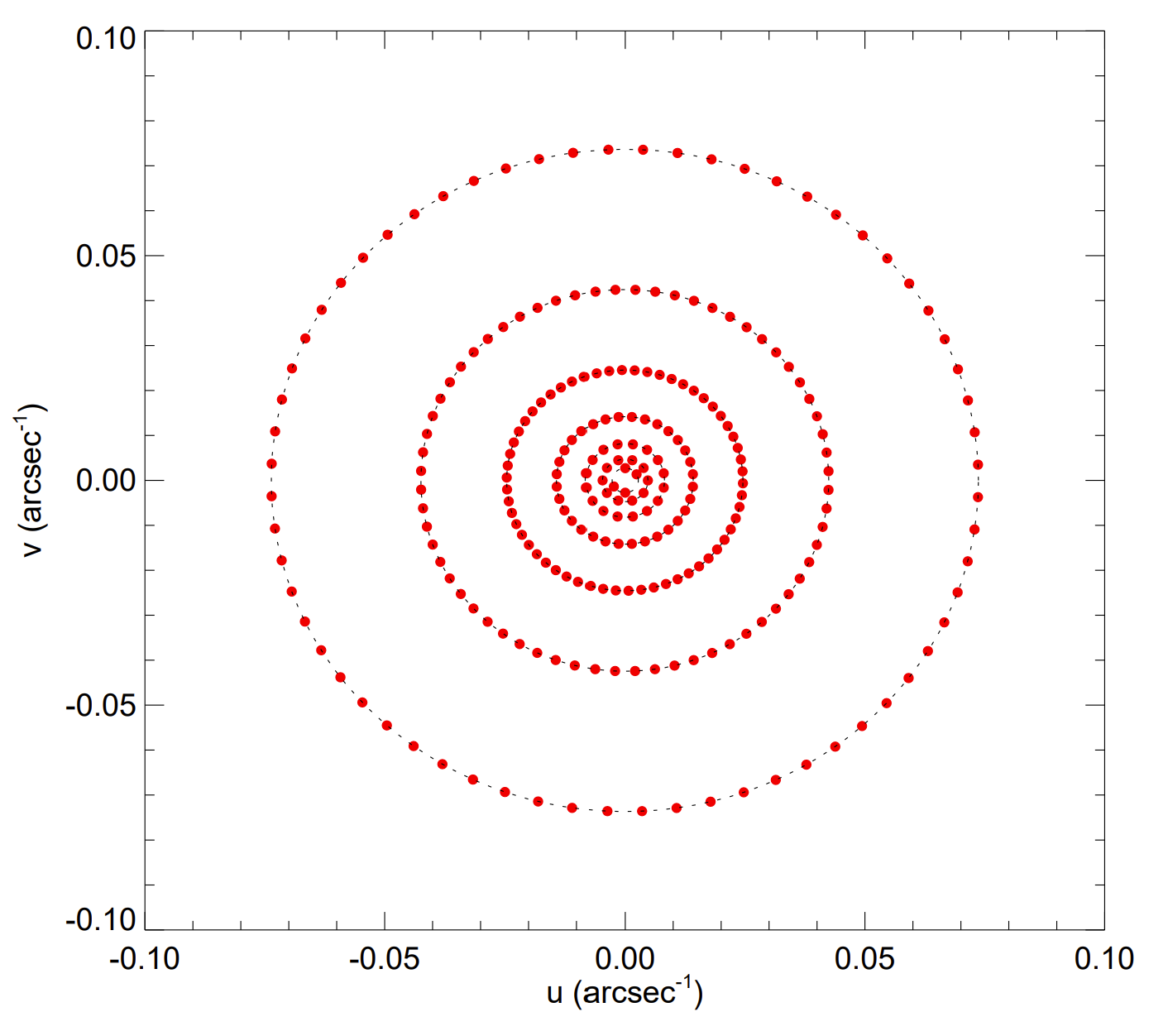}
    \caption{Frequency points sampled by the
RHESSI instrument. The sampled points are located on 7 circles (only detectors 3 to 9 are used), with increasing radii. The number of samples is about $250$; it depends on the signal to noise ratio and changes flare to flare.}
    \label{fig:rhessicoverage}
\end{figure}

The RHESSI visibilities are measured at specific energy channels, i.e., such energy ranges will play the role of $\beta$ in \eqref{eq:schr}. In this specific context, \eqref{eq:schr} relates the photon and electron visibilities, indeed as shown in \cite{piana2007electron}, for a fixed $\boldsymbol{u}$, the link between $V(\boldsymbol{u};\epsilon)$ and $w(\boldsymbol{u};e)$, which respectively are the photon and electron visibilities, is given by:
\begin{equation*}\label{V-W}
V(\boldsymbol{u};\epsilon)=\frac{1}{4\pi R^2} \int_{e=\epsilon}^{\infty} w(\boldsymbol{u};e)\eta(\epsilon,e) 
\,de \, ,
\end{equation*}
where $R$ is the distance from the source to the instrument,  $\epsilon$ and $e$ are {photon} and {electron energy, respectively}, and $\eta(\epsilon,e)$ is the bremsstrahlung cross-section, that gives the probability that a photon whose energy is $\epsilon$ is emitted by an electron whose energy is $e \ge \epsilon$. Considering the nature of the signal we are working with, i.e. photon counts, the solution map is always non-negative, and this implies that the positivity constraints in the Landweber iterations works in this setting too.

\begin{corollary}
   Under the same assumptions of Proposition \ref{th:g-tildeg}, and fixing $e>\epsilon_2>\epsilon_1$, if $\epsilon_1 \longrightarrow \epsilon_2$ then $|V(\boldsymbol{u},\epsilon_1)-V(\boldsymbol{u},\epsilon_2)| \longrightarrow 0$.
\end{corollary}

\begin{proof}
Following the same steps of Proposition \ref{th:g-tildeg}, we get:
\begin{equation*}
\begin{split}
  |V(\boldsymbol{u},\epsilon_1)-V(\boldsymbol{u},\epsilon_2)| & \leq \dfrac{1}{4 \pi R^2}  \left ( \int_{e=\epsilon_1}^{\epsilon_2} \left |w(\boldsymbol{u}, e) \eta(\epsilon_1,e) \right|de \right. \\
  & \left. + \int_{e=\epsilon_2}^{\infty} \left |w(\boldsymbol{u}, e) || \eta(\epsilon_1,e) -\eta(\epsilon_2,e)  \right |de \right ).
\end{split}
\end{equation*}
Then, as shown in \cite{massone2008regularized}, if $\epsilon_1 \longrightarrow \epsilon_2$, then $\int_{\epsilon_1}^{\epsilon_2} |w(\boldsymbol{u}, e) \eta(\epsilon_1,e)|de \longrightarrow 0$, and the thesis follows.
\end{proof}

This suggests to use as scaling function dependent on the energy range. Precisely, let us suppose to have a reconstructed image $\tilde{I}_{\epsilon_2}:=\tilde{I}_{\epsilon_2}(\cdot, \epsilon_2)$ at the energy range $\epsilon_2$. Then, our interpolant at the sought energy range $\epsilon_1$ 
will be constructed with VSKs, and as scaling function $\tilde{V}_{\epsilon_2}:= \tilde{V}(\cdot, \epsilon_2)$ we take: 
$$\tilde{V}_{\epsilon_2} (\boldsymbol{u}) =  ({\cal F}\tilde{I}_{\epsilon_2})(\boldsymbol{u}).$$
Then, we recover $\tilde{I}_{\epsilon_1}$ by approximating via projected Landweber iterative scheme the following inverse problem:
$$
P_{V_{\epsilon_1}}^{\tilde{V}_{\epsilon_2}}(\boldsymbol{u}) = ({\cal F}\tilde{I}_{\epsilon_1})(\boldsymbol{u}).
$$

\subsection{Numerical results}
We test our procedure on a flare observed by RHESSI on July 03 2002, considering the time window between 02:10:13 and 02:11:13 UT. RHESSI spectral resolution was $1-2$ keV below $1$ MeV \cite{piana2022hard}, therefore, we extend what discussed in the previous sections, regarding the $\beta$ parameter connected to the hardware parameter, to sequencies of contiguous energy bins. Specifically, we have considered $2$ $3$-keV-wide channels, starting from $25-28$ keV, and $9$ $2$-keV-wide channels, from $20-22$ keV to $4-6$ keV. For this event RHESSI provided $n = 248$ visibilities, at each energy channel. Figure \ref{fig:highlow} shows the reconstruction of the event in the considered energy intervals, using the interpolation/extrapolation procedure. The reconstruction used to trigger the scheme is the one provided by the Maximum Entropy Method MEM\_GE (\cite{Massa_2020}), in the energy range $25-28$ keV. From the energy range $22-25$ keV, each map is obtained by using as scaling function in the interpolation step the discrete Fourier trasform of the map retrieved in the previous (higher) energy channel. Therefore, given a succession of energy channels $\{\epsilon_{\ell}\}_{\ell=1}^L$, where $\epsilon_{\ell+1} < \epsilon_{\ell}$, we consider a sequence of scaling functions, $\{\tilde{V}_{\epsilon_{\ell}}\}_{\ell=1}^L$, and interpolants, to iteratively approximate $\tilde{I}_{\epsilon_{\ell+1}}$ via the relation:
\[
P_{V_{\epsilon_{\ell+1}}}^{\tilde{V}_{\epsilon_{\ell}}}(\boldsymbol{u}) = ({\cal F}\tilde{I}_{\epsilon_{\ell+1}})(\boldsymbol{u})
\]
where 
\[
\tilde{V}_{\epsilon_{\ell}} (\boldsymbol{u}) =  ({\cal F}\tilde{I}_{\epsilon_{\ell}})(\boldsymbol{u}), \quad \ell = 1, \ldots, L.
\]
We refer the reader to \cite{volpara2024ris} for further details.

\begin{figure}[htbp]
\centering
\includegraphics[width=1\textwidth]{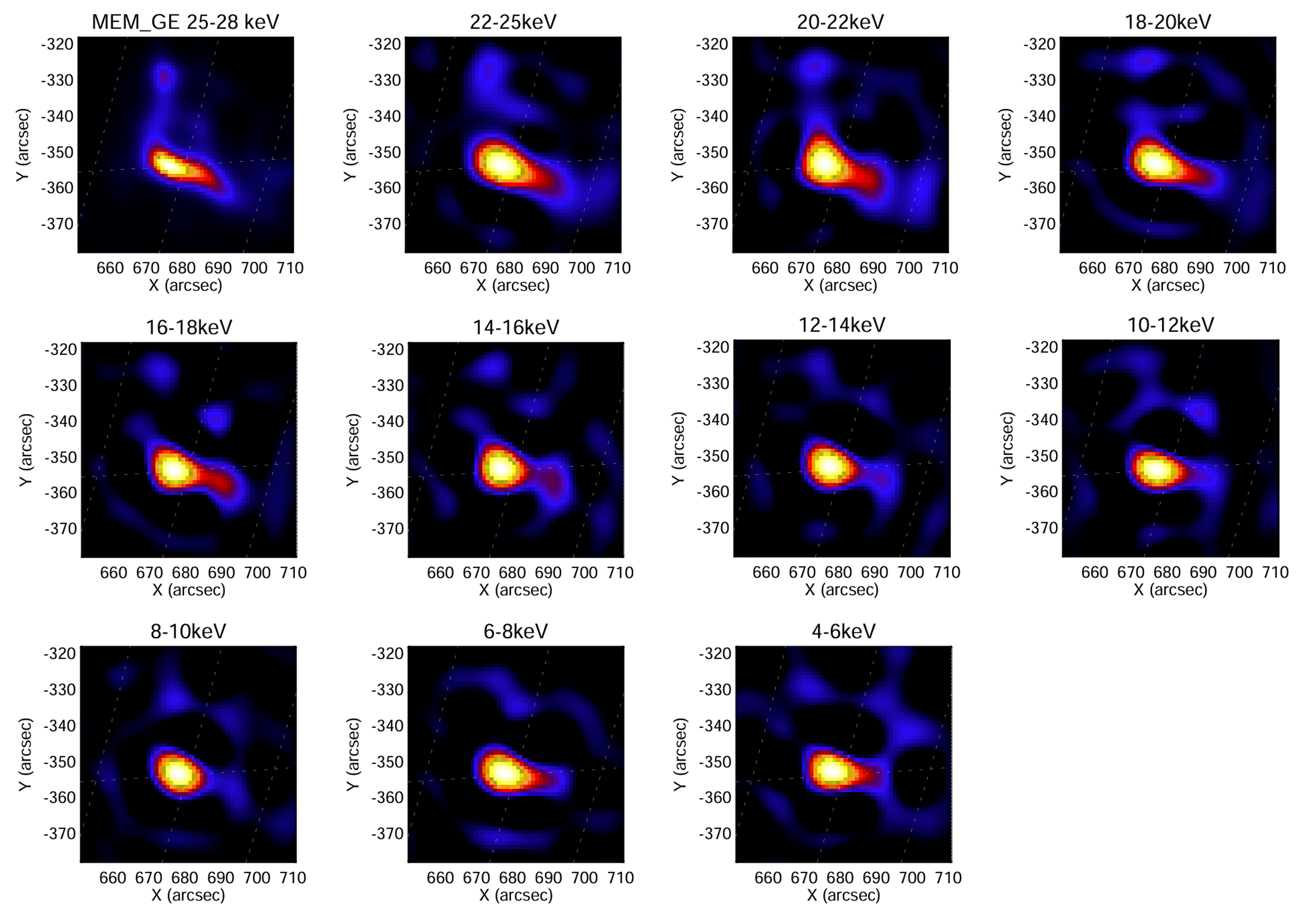}
\caption{Reconstructions of July 03, 2002 event, in the time range 02:10:13 -02:11:13 UT for the energy intervals shown, obtained by the interpolation/extrapolation iterative method, using the MEM\_GE reconstruction at the top left panel as triggering image.}\label{fig:highlow}
\end{figure}

An advantage of this approach is that visibilities in continuous energy channels are correlated. Thus, as shown in Figure \ref{fig:spectra}, regularized photon spectra can be obtained: local count spectra derived from the regularized maps shown in Figure \ref{fig:highlow} are compared with those obtained from maps reconstructed using MEM\_GE, when apply independently at each energy channel. The interpolation/extrapolation procedure ensures numerical stability and reduces uncertainties (obtained by a confidence strip approach), as evidenced by the black error bars, which are consistently smaller than those associated with the unregularized (uncorrelated) spectra (in pink). 

\begin{figure}
    \centering
    \includegraphics[width=0.9\linewidth]
    {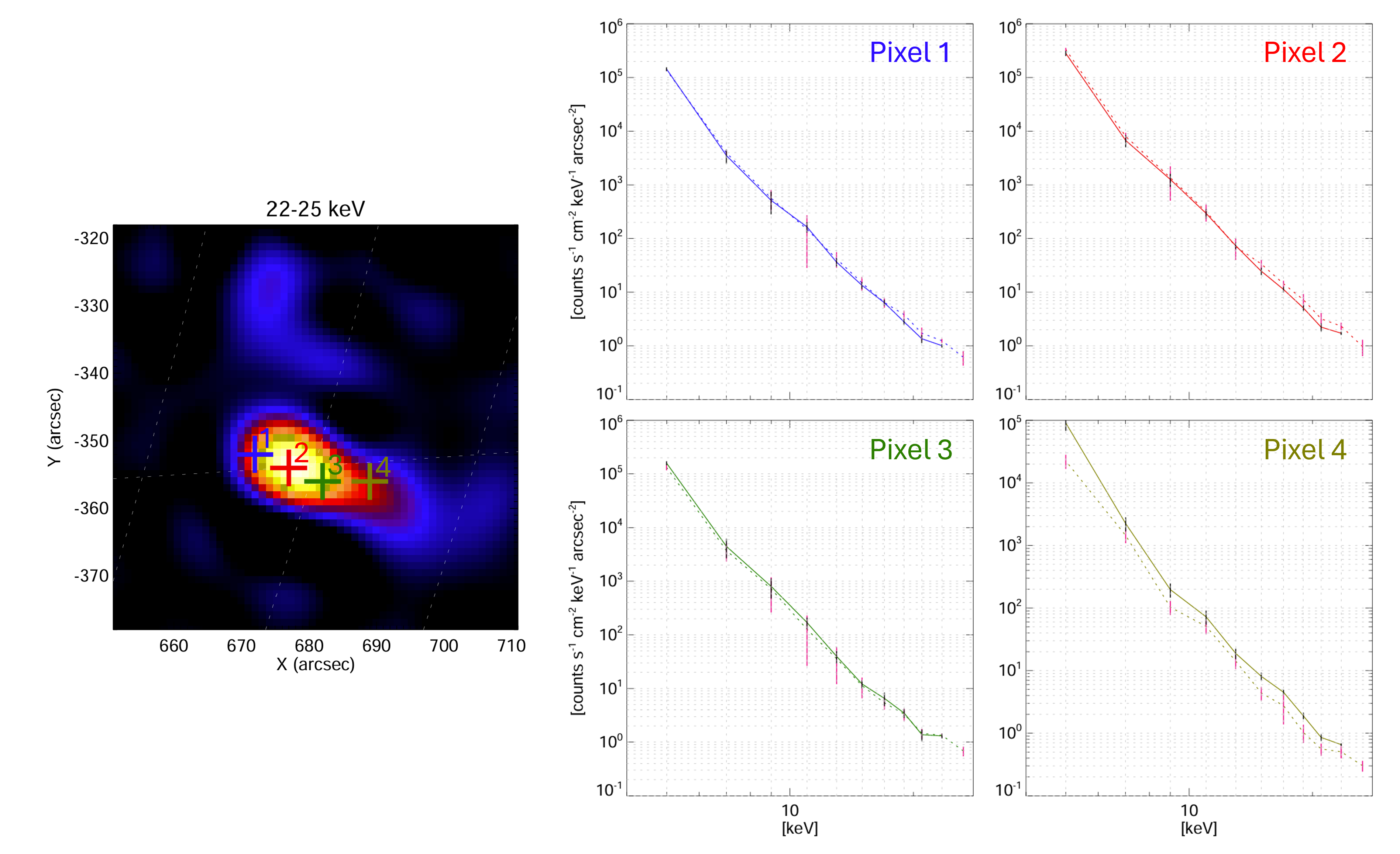}
    \caption{Left panel: selected pixels are indicated with colored crosses. Right panel: corresponding pixel-wise spectra obtained from photon images, in the case of iterative interpolation/extrapolation precedure (solid line) and MEM\_GE (dashed line). The pixels selected in the left panel and their respective spectra are indicated with the same color.}
    \label{fig:spectra}
\end{figure}

Finally, the theoretical results given in Subsection \ref{Landweber} are confirmed by Figures \ref{stoprule} and \ref{prop}, obtained by setting the relaxation parameter for Landweber iteration $\tau =0.2$, and the shape parameter for the RBF $\gamma = 0.02$. Specifically, each panel of Figure \ref{stoprule} corresponds to a specific energy channel. 
For each iteration, the red line shows the global distance between two subsequent steps of the Landweber iterative scheme, while the black line shows the global distance between the maps at each iteration of the Landweber process and the initial interpolants, taking into account the contribution related to $\mathcal{P}_-I^{(k+1)}$, as defined in the left-hand side of Theorem \ref{th:convergence}. 

Figure \ref{prop} shows the result established in Proposition \ref{prop land}. It is important to note that no information is available outside the compact Fourier domain $D$, and thus, the first term on the right-hand side is not taken into account. Additionally, the same figure demonstrates that the negative contribution in the Landweber scheme (associated with $\mathcal{P}_-I^{(k+1)}=\tilde{\tilde{I}}^{(k+1)}$) is negligible compared to the map obtained after applying the positivity constraint (i.e., $\tilde{I}^{(k+1)}$), blue line in Figure.

\begin{figure}[htbp]
\centering
\includegraphics[width=1\textwidth]
{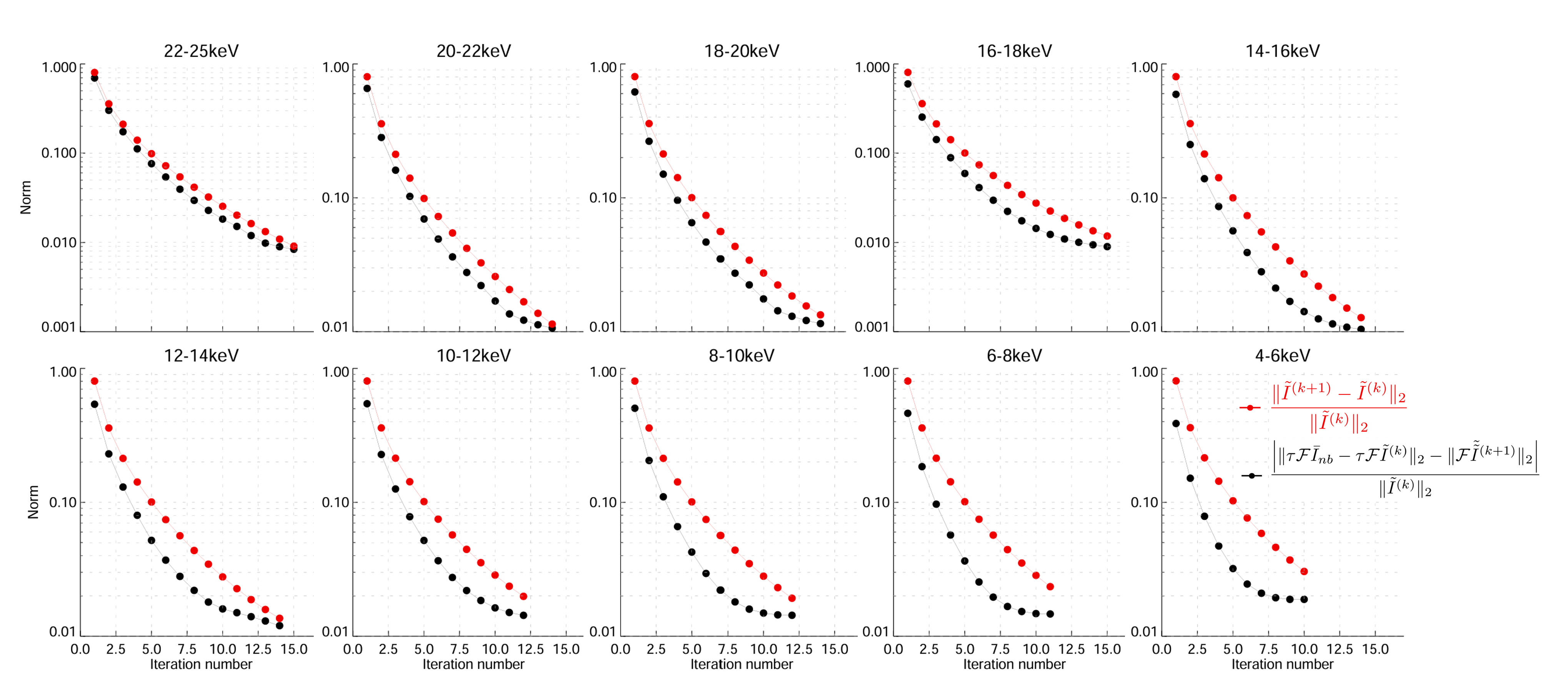}
\caption{Numerical results for Theorem \ref{th:convergence} in the case of RHESSI data. 
}\label{stoprule}
\end{figure}

\begin{figure}[htbp]
\centering
\includegraphics[width=0.9\textwidth]{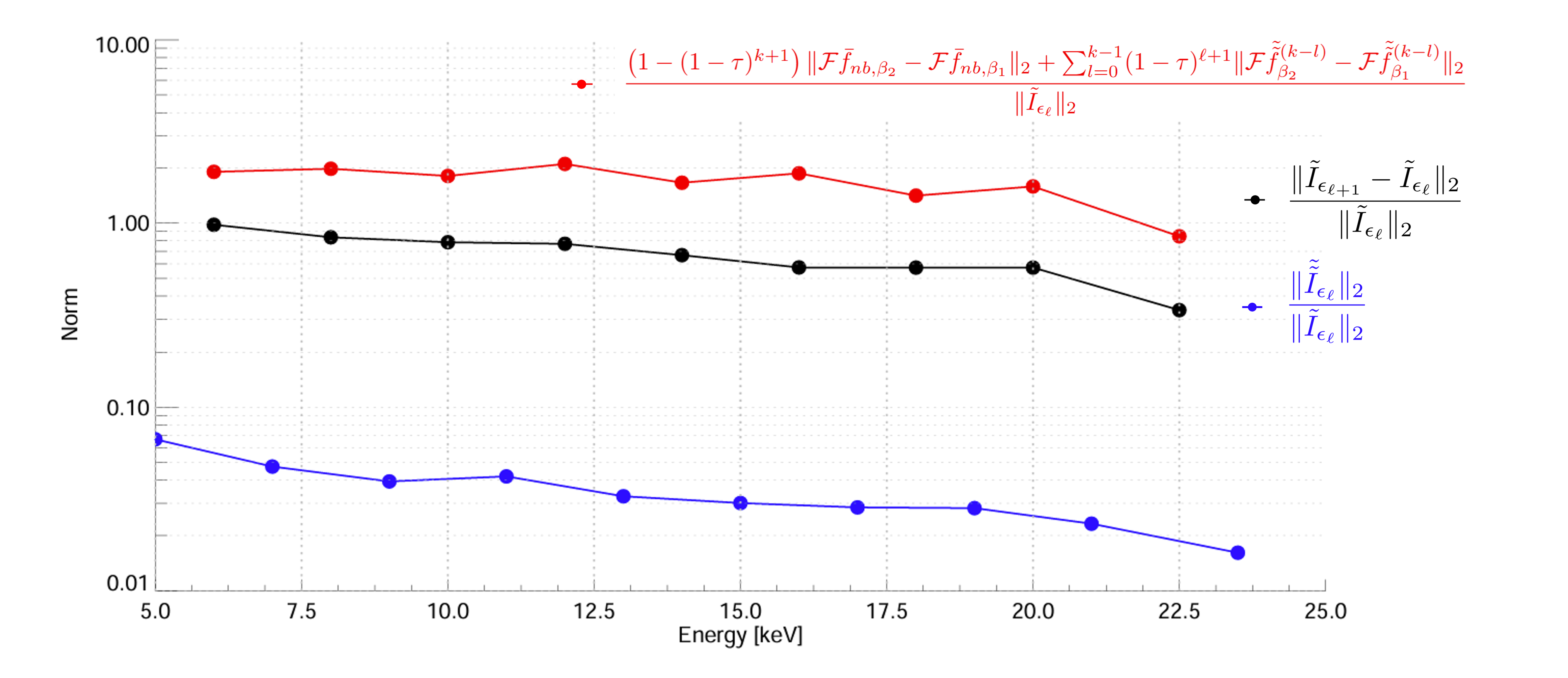}
\caption{Numerical results for the Proposition \ref{prop land} in the case of RHESSI data. }\label{prop}
\end{figure}

\section{Conclusions}
\label{conclusions}

In this study, we have advanced the theoretical understanding of Fourier transform inversion from limited data by employing Variably Scaled Kernels (VSKs). The $L_2$ error bounds we established, which are linked to the choice of the scaling function, provide a solid theoretical foundation for selection of the VSK basis. 
Additionally, our findings on the Landweber iterative scheme emphasize the importance of selecting an appropriate scaling function, which can be tailored to specific hardware or data conditions. For instance, in the case of RHESSI, we select a VSK basis that depends on the energy ranges, allowing for smoother reconstructions along the energy levels.

These insights could guide the development of more refined VSK-based algorithms for sparser imaging systems, especially given the growing interest in nano-satellites that provide only a few observations. Overall, this work lays the groundwork for further exploration of VSKs in a range of inverse problem settings, such as medical imaging.

\section*{Acknowledgments}
Anna Volpara and Emma Perracchione kindly acknowledge the support of the Fondazione Compagnia di San Paolo within the framework of the Artificial Intelligence Call for Proposals, AIxtreme project (ID Rol: 71708) and the GNCS-IN$\delta$AM PIANIS project ``Problemi Inversi e Approssimazione Numerica in Imaging Solare''. Anna Volpara is also supported by the ``Accordo ASI/INAF Solar Orbiter: Supporto scientifico per la realizzazione degli strumenti Metis, SWA/DPU e STIX nelle Fasi D-E''. Emma Perracchione further acknowledge the support of the GOSSIP project (``Greedy Optimal Sampling for Solar Inverse Problems'') – funded by the Ministero dell’Universit\`a e della Ricerca (MUR) -  within the PRIN 2022 program, CUP: E53C24002330001. This research has been partly performed in the framework of the MIUR Excellence Department Project awarded to Dipartimento di Matematica, Università di Genova, CUP D33C23001110001. 
{All the authors thank Michele Piana and Frank Filbir for discussions. }

\bibliographystyle{unrst}

\end{document}